\documentclass[10pt]{article}

\usepackage[T1]{fontenc}
\usepackage{lmodern}
\usepackage{microtype}
\usepackage[margin=0.9in]{geometry}
\usepackage{amsmath,amssymb,amsthm}
\usepackage{mathrsfs}
\usepackage{tikz}
\usetikzlibrary{arrows.meta}
\usepackage[hidelinks]{hyperref}
\hypersetup{
 pdftitle={An Ellipse Criterion for Exact Nonuniqueness in the Planar Interior Radon Problem},
 pdfauthor={Christian H{\"a}gg},
 pdfsubject={The interior problem for the planar Radon transform},
 pdfkeywords={Radon transform, interior problem, restricted data, convex geometry, ellipse}
}

\newcommand{\R}{\mathbb R}
\newcommand{\C}{\mathbb C}
\newcommand{\T}{\mathbb T}
\newcommand{\Rad}{\mathcal R}
\newcommand{\supp}{\operatorname{supp}}
\newcommand{\co}{\operatorname{co}}
\newcommand{\tr}{\operatorname{tr}}
\newcommand{\ran}{\operatorname{ran}}

\newcommand{\Real}{\operatorname{Re}}
\newcommand{\Imag}{\operatorname{Im}}

\newtheorem{theorem}{Theorem}
\newtheorem{lemma}[theorem]{Lemma}
\newtheorem{proposition}[theorem]{Proposition}
\newtheorem{corollary}[theorem]{Corollary}
\theoremstyle{remark}
\newtheorem{remark}[theorem]{Remark}

\title{An Ellipse Criterion for Exact Nonuniqueness\\ in the Planar Interior Radon Problem}
\author{Christian H\"agg\\
Department of Mathematics, Stockholm University\\
\texttt{hagg@math.su.se}}
\date{}

\begin{document}
\maketitle

\begin{abstract}
We characterize exact nonuniqueness in the planar interior Radon problem for arbitrary pairs of open convex sets. There exists a nonzero smooth function compactly supported in the first set whose integral over every line meeting the second set vanishes if and only if an ellipse contains the closure of the second set and is compactly contained in the first set. The same criterion holds for a nonzero $L^1$ function whose essential support is compactly contained in the first set and whose Radon transform vanishes almost everywhere on those lines. For concentric open squares, the criterion gives the sharp threshold $1/\sqrt2$ for the inner-to-outer half-side ratio. This yields counterexamples to Conjecture~1.2 of \cite{BomanEllipsoidII} and to Theorem~40.1 of \cite{BomanIPMS}.
\end{abstract}

The interior Radon transform is a classical restricted-data problem; see \cite{Maass,Natterer} and, for broader background on limited-data tomography, \cite{KrishnanQuinto}. It asks whether a function can be recovered from its integrals over the lines meeting a prescribed subdomain. In the two-domain formulation considered here, exact nonuniqueness occurs when a nonzero function supported in one prescribed domain has zero line integrals over every line meeting the other prescribed domain.

For concentric disks such examples are classical \cite[Section~VI.4]{Natterer}, and affine covariance gives the same conclusion whenever an ellipse can be placed between the two domains. The question arose in connection with support and regularity results for Radon transforms \cite{BomanEllipsoidI,BomanEllipsoidII,BomanRegularity}. Universal nonuniqueness for every nested convex pair was conjectured in \cite[Conjecture~1.2]{BomanEllipsoidII}, and the centrally symmetric case was asserted in \cite[Theorem~40.1]{BomanIPMS}. Our main result gives the exact geometric criterion in the plane: the intervening ellipse is not merely sufficient, but necessary, and it need not share a prescribed center with either domain. In particular, Corollary~\ref{cor:squares} and Corollary~\ref{cor:closed} provide concentric-square counterexamples to those assertions. Remark~\ref{rem:boman} isolates the parameter dependence in the argument for \cite[Theorem~40.1]{BomanIPMS}.

Unless specified otherwise, functions are complex-valued and smoothness refers to real variables. Let $\mathbb S^1=\{\omega\in\R^2:|\omega|=1\}$ carry arc-length measure $d\omega$, and give $\mathbb S^1\times\R$ the product measure $d\omega\,dp$. For $\omega=(\omega_1,\omega_2)$, let $\omega^\perp=(-\omega_2,\omega_1)$. For $\omega\in \mathbb S^1$ and $p\in\R$, write
\[
 L(\omega,p)=\{x\in\R^2:x\cdot\omega=p\},
 \qquad
 \Rad u(\omega,p)=\int_{L(\omega,p)}u\,ds,
\]
where $ds$ denotes Euclidean arc length. If $u\in L^1(\R^2)$, Fubini's theorem in the coordinates $x=p\omega+t\omega^\perp$ defines $\Rad u$ as an $L^1(\mathbb S^1\times\R)$-class, independently of the representative of $u$, with
\[
 \|\Rad u\|_{L^1(\mathbb S^1\times\R)}
 \le 2\pi\|u\|_{L^1(\R^2)}.
\]
For such $u$, $\supp u$ is the complement of the largest open set on which $u=0$ almost everywhere; for continuous $u$, this is the usual closed support. We write $\Omega_0\Subset\Omega_1$ if $\overline{\Omega_0}$ is compact and contained in $\Omega_1$; thus $C_c^\infty(D)=\{u\in C^\infty(\R^2):\supp u\Subset D\}$. For a set $K\subseteq\R^2$, let $\co(K)$ denote its convex hull. We write $\T=\{q\in\C:|q|=1\}$, and for $r>0$ set $\mathbb D_r=\{x\in\R^2:|x|<r\}$.

For a complex matrix or a bounded operator $T$ on a complex Hilbert space, we set
\[
 \Real T=\frac{T+T^*}{2},
 \qquad
 \Imag T=\frac{T-T^*}{2i}.
\]
All Hilbert-space inner products are linear in the first variable. For a complex Hilbert space $\mathcal H$, $\mathcal B(\mathcal H)$ denotes the algebra of bounded linear operators on $\mathcal H$; $\sigma(\cdot)$ denotes the spectrum, $\ominus$ the orthogonal difference of closed subspaces, and $\simeq$ unitary equivalence. A function $u$ on $\R^2$ is \emph{centrally symmetric about} $c\in\R^2$ if $u(c+x)=u(c-x)$ for every $x$ (for almost every $x$ when $u$ is an $L^1$-class).

For a nonempty bounded convex set $K$, define its projection endpoints by
\begin{equation}\label{eq:endpoints}
 a_K(\omega)=\inf_{x\in K}x\cdot\omega,
 \qquad
 b_K(\omega)=\sup_{x\in K}x\cdot\omega.
\end{equation}
Thus $b_K$ is the support function of $K$ and $a_K(\omega)=-b_K(-\omega)$. We use $a_{\overline K}=a_K$ and $b_{\overline K}=b_K$ for nonempty bounded $K$. For nonempty compact convex sets $K_-$ and $K_+$, the standard support-function criterion may be written as
\begin{equation}\label{eq:support-inclusion}
 K_-\subseteq K_+
 \quad\Longleftrightarrow\quad
 a_{K_+}\le a_{K_-}\le b_{K_-}\le b_{K_+}.
\end{equation}

By an \emph{ellipse} we mean a closed, nondegenerate elliptic disk
\[
 E=c+Q^{1/2}\overline{\mathbb D_1},
\]
where $c\in\R^2$ and $Q$ is a real symmetric positive definite $2\times2$ matrix. One has
\begin{equation}\label{eq:ellipse-endpoints}
 a_E(\omega)=c\cdot\omega-(\omega^{\mathsf T}Q\omega)^{1/2},
 \qquad
 b_E(\omega)=c\cdot\omega+(\omega^{\mathsf T}Q\omega)^{1/2}.
\end{equation}
We call $c$ the center of $E$.

\begin{theorem}[Ellipse criterion]\label{thm:main}
Let $D_0,D\subseteq\R^2$ be open convex sets. The following are equivalent.
\begin{enumerate}
\item[(i)] There is a nonzero $u\in C_c^\infty(D)$ such that
\begin{equation}\label{eq:smooth-vanishing}
 \Rad u(\omega,p)=0
 \quad\text{whenever }L(\omega,p)\cap D_0\ne\varnothing.
\end{equation}
\item[(ii)] There is a nonzero $u\in L^1(\R^2)$ with $\supp u\Subset D$ such that
\begin{equation}\label{eq:ae-vanishing}
 \Rad u(\omega,p)=0
 \quad\text{for almost every }(\omega,p)\text{ with }
 L(\omega,p)\cap D_0\ne\varnothing.
\end{equation}
\item[(iii)] There is an ellipse $E$ such that
\begin{equation}\label{eq:ellipse-sandwich}
 \overline{D_0}\subseteq E\Subset D.
\end{equation}
\end{enumerate}
Each condition forces $D_0\Subset D$.
If (iii) holds with center $c$, the function in (i) may be chosen real-valued, centrally symmetric about $c$, and strictly negative on an open neighborhood of $E$; in particular, $E\subseteq\supp u$.
\end{theorem}

\begin{remark}[Connected data sets]\label{rem:connected}
The theorem also holds when $D_0$ is connected and not contained in a line, while $D$ remains open and convex. Indeed, connectedness implies that $D_0$ and $\co(D_0)$ meet exactly the same lines. The latter has nonempty interior and $\overline{\operatorname{int}\co(D_0)}=\overline{\co(D_0)}$. Apply necessity to this interior and Proposition~\ref{prop:sufficiency} to the resulting ellipse. Thus $D_0$ may, for example, be a circle of positive radius or the boundary of a nondegenerate polygon.
\end{remark}

When $\varnothing\ne D_0\Subset D$ and $D$ is bounded, the geometric condition is a convex feasibility problem in five real variables, with constraints indexed by $\mathbb S^1$; see Figure~\ref{fig:criterion}.

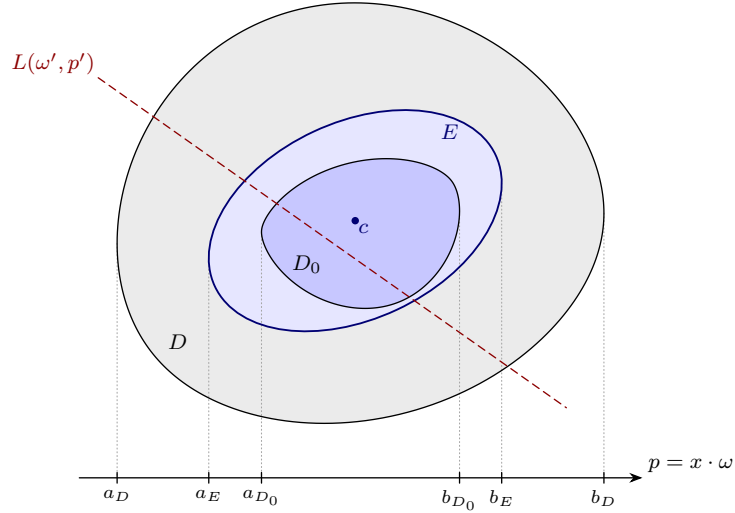
\begin{figure}[t]
\centering
\begin{tikzpicture}[
 line cap=round, line join=round,
 declare function={
  hD(\t)=3.0+0.30*cos(\t)-0.10*sin(\t)+0.22*cos(2*\t)+0.10*sin(2*\t)
         +0.07*cos(3*\t)-0.05*sin(3*\t);
  dhD(\t)=-0.30*sin(\t)-0.10*cos(\t)-0.44*sin(2*\t)+0.20*cos(2*\t)
          -0.21*sin(3*\t)-0.15*cos(3*\t);
  hI(\t)=1.15+0.42*cos(\t)-0.28*sin(\t)+0.16*cos(2*\t)+0.07*sin(2*\t)
         -0.05*cos(3*\t)+0.04*sin(3*\t);
  dhI(\t)=-0.42*sin(\t)-0.28*cos(\t)-0.32*sin(2*\t)+0.14*cos(2*\t)
          +0.15*sin(3*\t)+0.12*cos(3*\t);
 }]
\path[save path=\pathD]
 plot[domain=0:357,samples=120,smooth cycle]
 ({hD(\x)*cos(\x)-dhD(\x)*sin(\x)},{hD(\x)*sin(\x)+dhD(\x)*cos(\x)});
\path[save path=\pathI]
 plot[domain=0:357,samples=120,smooth cycle]
 ({hI(\x)*cos(\x)-dhI(\x)*sin(\x)},{hI(\x)*sin(\x)+dhI(\x)*cos(\x)});
\fill[black!8, use path=\pathD];
\fill[blue!10] (0.30,-0.15) ellipse[x radius=2.05, y radius=1.30, rotate=25];
\fill[blue!22, use path=\pathI];
\foreach \x/\y in {-2.85/-0.45, -1.6375/-0.6467, -0.94/-0.30,
                   1.68/-0.02, 2.2375/0.3467, 3.59/-0.05}
 \draw[gray!70, densely dotted, line width=0.4pt] (\x,\y) -- (\x,-3.55);
\draw[-{Stealth[length=5pt]}, line width=0.5pt]
 (-3.35,-3.55) -- (4.10,-3.55)
 node[above right, inner sep=2pt] {\footnotesize $p=x\cdot\omega$};
\foreach \x/\l in {-2.85/{a_D}, -1.6375/{a_E}, -0.94/{a_{D_0}},
                   1.68/{b_{D_0}}, 2.2375/{b_E}, 3.59/{b_D}}
 \draw[line width=0.5pt] (\x,-3.48) -- (\x,-3.62)
   node[below, inner sep=2.5pt] {\scriptsize $\l$};
\draw[black, line width=0.5pt, use path=\pathD];
\draw[blue!45!black, line width=0.7pt]
 (0.30,-0.15) ellipse[x radius=2.05, y radius=1.30, rotate=25];
\draw[black, line width=0.5pt, use path=\pathI];
\draw[red!55!black, densely dashed, line width=0.5pt]
 (-3.10,1.74) -- (3.093,-2.6236);
\node[red!55!black, above left, inner sep=1pt] at (-3.10,1.74)
 {\footnotesize $L(\omega',p')$};
\fill[blue!45!black] (0.30,-0.15) circle (1.4pt);
\node[blue!45!black, below right, inner sep=1pt] at (0.30,-0.15) {\footnotesize $c$};
\node at (-2.05,-1.75) {\footnotesize $D$};
\node[blue!45!black] at (1.55,1.02) {\footnotesize $E$};
\node at (-0.34,-0.72) {\footnotesize $D_0$};
\end{tikzpicture}
\caption{The ellipse criterion of Theorem~\ref{thm:main}: an ellipse $E$ centered at $c$ satisfies $\overline{D_0}\subseteq E\Subset D$. For the direction $\omega$ of the horizontal projection axis, the dotted guides give $a_D(\omega)<a_E(\omega)\le a_{D_0}(\omega)<b_{D_0}(\omega)\le b_E(\omega)<b_D(\omega)$, the interval sandwich used in Remark~\ref{rem:feasibility}. The dashed line is $L(\omega',p')$ for a second direction $\omega'$; it meets $D_0$, so every witness $u$ in condition~\textup{(i)} has vanishing line integral over it, $\Rad u(\omega',p')=0$.}
\label{fig:criterion}
\end{figure}

\begin{remark}[Five-variable feasibility]\label{rem:feasibility}
Suppose that $\varnothing\ne D_0\Subset D$ and $D$ is bounded, and write
\[
 a_0=a_{D_0},\quad b_0=b_{D_0},\quad
 a_1=a_D,\quad b_1=b_D.
\]
Condition~(iii) of Theorem~\ref{thm:main} is equivalent to the existence of $c\in\R^2$ and a real symmetric matrix $B$ such that, for every $\omega\in \mathbb S^1$, the monic quadratic
\begin{equation*}
 F_\omega(p)=p^2-2(c\cdot\omega)p+\omega^{\mathsf T}B\omega
\end{equation*}
satisfies
\begin{equation}\label{eq:linear-feasibility}
 F_\omega(a_1(\omega))>0,
 \qquad F_\omega(a_0(\omega))\le0,
 \qquad F_\omega(b_0(\omega))\le0,
 \qquad F_\omega(b_1(\omega))>0.
\end{equation}
Indeed, if $E=c+Q^{1/2}\overline{\mathbb D_1}$, take $B=cc^{\mathsf T}-Q$; then $F_\omega$ vanishes at the two endpoints in \eqref{eq:ellipse-endpoints}, and \eqref{eq:linear-feasibility} is exactly the interval sandwich
\[
 [a_0(\omega),b_0(\omega)]
 \subseteq[a_E(\omega),b_E(\omega)]
 \subseteq(a_1(\omega),b_1(\omega)).
\]
Conversely, suppose \eqref{eq:linear-feasibility} holds. Since $D_0$ is nonempty and open and $D_0\Subset D$,
\[
 a_1(\omega)<a_0(\omega)<b_0(\omega)<b_1(\omega)
 \qquad(\omega\in \mathbb S^1).
\]
Because $F_\omega$ is monic and nonpositive at the distinct points $a_0(\omega)$ and $b_0(\omega)$, it has two distinct real roots
\[
 r_-(\omega)<r_+(\omega).
\]
Set
\[
 Q:=cc^{\mathsf T}-B.
\]
Then
\[
 F_\omega(p)=(p-c\cdot\omega)^2-\omega^{\mathsf T}Q\omega,
\]
and the existence of two distinct roots gives $\omega^{\mathsf T}Q\omega>0$ for every $\omega\in \mathbb S^1$. Thus $Q$ is positive definite, and $r_-(\omega),r_+(\omega)$ are the projection endpoints of
\[
 E=c+Q^{1/2}\overline{\mathbb D_1}.
\]
The signs in \eqref{eq:linear-feasibility}, together with the strict ordering above, give
\[
 a_1(\omega)<r_-(\omega)\le a_0(\omega)<b_0(\omega)
 \le r_+(\omega)<b_1(\omega).
\]
The support-function criterion \eqref{eq:support-inclusion} therefore yields
\[
 \overline{D_0}\subseteq E\subseteq\overline D.
\]
If $E$ met $\partial D$, a supporting line to $\overline D$ at a point of contact would contradict one of the strict outer endpoint inequalities. Hence $E\subseteq D$, and compactness gives $E\Subset D$. Thus \eqref{eq:ellipse-sandwich} holds. Each inequality in \eqref{eq:linear-feasibility} is affine in the five real variables consisting of the two entries of $c$ and the three independent entries of $B$.
\end{remark}

For $D=(-1,1)^2$ and $D_0=(-\delta,\delta)^2$, the criterion gives the sharp threshold $\delta<1/\sqrt2$; see Corollary~\ref{cor:squares}. Sufficiency is obtained by transporting the classical disk construction through an affine map. Necessity has four stages. First, the coordinates
\[
 q=e^{2i\theta},\qquad z=pe^{i\theta}
\]
identify the unoriented line $(\theta,p)\sim(\theta+\pi,-p)$ without imposing parity on the unknown function. The classical Helgason--Ludwig moment conditions \cite{Helgason,Ludwig} then give an annihilating measure on a compact subset of the unoriented-line space $\Sigma$ defined in Section~2. Second, a uniform-algebra argument produces a representing probability measure. Third, the coordinate multipliers generate a pure isometry and hence a vector-valued Hardy model in which the second coordinate has the form $A+qA^*$ for a bounded operator $A$ on an auxiliary Hilbert space (Lemma~\ref{lem:hardy-model}). Finally, a spectral-gap theorem (Theorem~\ref{thm:operator-ellipse}) for the self-adjoint pencil $\Real(e^{-i\theta}\mathbf A)$, applied to $\mathbf A=2A$, extracts a $2\times2$ matrix whose projection intervals form the required ellipse, thereby forcing the geometric nesting.

\section{Sufficiency}

We first record a self-contained version of the classical disk construction (cf.\ \cite[Section~VI.4]{Natterer}).

\begin{lemma}[Concentric disks]\label{lem:disks}
Let $0<r_{\mathrm{in}}<r_{\mathrm{out}}$. There is a nonzero real radial function $f\in C_c^\infty(\mathbb D_{r_{\mathrm{out}}})$ such that
\[
 \Rad f(\omega,p)=0\qquad (|p|\le r_{\mathrm{in}}),
 \qquad \overline{\mathbb D_{r_{\mathrm{in}}}}\subseteq\supp f,
\]
and $f$ is strictly negative on an open neighborhood of $\overline{\mathbb D_{r_{\mathrm{in}}}}$.
\end{lemma}

\begin{proof}
Choose $r_{\mathrm{in}}<r_*<r_{\mathrm{out}}$ and a nonzero nonnegative even function $g\in C_c^\infty(\R)$ with
\[
 \supp g\subseteq(-r_{\mathrm{out}},-r_*)\cup(r_*,r_{\mathrm{out}}).
\]
For $\rho\ge0$ set
\begin{equation}\label{eq:abel-profile}
 \Phi(\rho)=-\frac1\pi\int_\rho^\infty
       \frac{g'(t)}{\sqrt{t^2-\rho^2}}\,dt,
 \qquad f(x)=\Phi(|x|).
\end{equation}
This defines a smooth compactly supported radial function. Indeed, after the substitution $t=(\rho^2+y^2)^{1/2}$,
\[
 \Phi(\rho)=-\frac1\pi\int_0^\infty
 \frac{g'((\rho^2+y^2)^{1/2})}{(\rho^2+y^2)^{1/2}}\,dy.
\]
More explicitly, define
\[
 \psi(\lambda)=
 \begin{cases}
  g'(\sqrt\lambda)/\sqrt\lambda,&\lambda>0,\\
  0,&\lambda\le0.
 \end{cases}
\]
Because $g'$ vanishes near the origin, $\psi\in C_c^\infty(\R)$. Writing
\[
 \Psi(\lambda)=-\frac1\pi\int_0^\infty\psi(\lambda+y^2)\,dy,
\]
one has $\Phi(\rho)=\Psi(\rho^2)$. Compact support of $\psi$ justifies repeated differentiation under the integral, so $\Psi\in C^\infty(\R)$ and $f(x)=\Psi(|x|^2)$ is smooth. Since $\supp g\Subset(-r_{\mathrm{out}},r_{\mathrm{out}})$, one also has $\supp f\Subset\mathbb D_{r_{\mathrm{out}}}$.

For $0\le\rho<r_*$, put $G(y)=g((\rho^2+y^2)^{1/2})$. The function $G$ and its derivative vanish near $0$ and for sufficiently large $y$. Integration by parts on a compact interval away from $0$, using
\[
 G'(y)=y\psi(\rho^2+y^2),\qquad
 \left(-\frac1y\right)'=\frac1{y^2},
\]
expresses the integral of $\psi(\rho^2+y^2)$ as the integral of $G(y)/y^2$. The latter is strictly positive, since $g$ is nonnegative, even, and nonzero. Hence $\Phi(\rho)<0$.
Thus $f<0$ on $\mathbb D_{r_*}$, which is an open neighborhood of $\overline{\mathbb D_{r_{\mathrm{in}}}}$; in particular, $\overline{\mathbb D_{r_{\mathrm{in}}}}\subseteq\supp f$.

To compute the Radon transform, fix $p\in\R$ and set $H(s,y)=\psi(p^2+s^2+y^2)$. This is a continuous compactly supported function on $\R^2$, so Fubini's theorem applies. Evenness in $y$ and planar polar coordinates give
\begin{align*}
 \Rad f(\omega,p)
 &=\int_\R\Psi(p^2+s^2)\,ds\\
 &=-\frac1{2\pi}\int_{\R^2}H(s,y)\,ds\,dy\\
 &=-\int_0^\infty r\psi(p^2+r^2)\,dr
 =g(|p|)=g(p).
\end{align*}
The penultimate equality follows from the fundamental theorem of calculus applied to $r\mapsto g((p^2+r^2)^{1/2})$, whose derivative for $r>0$ is $r\psi(p^2+r^2)$ and whose limit at infinity is zero. The final equality uses evenness of $g$. The assertions follow because $g\ne0$ and $g=0$ on $[-r_{\mathrm{in}},r_{\mathrm{in}}]$.
\end{proof}

Sufficiency requires only an ellipse compactly contained in an open set.

\begin{proposition}\label{prop:sufficiency}
Let $D\subseteq\R^2$ be open and let $E\Subset D$ be an ellipse. There is a nonzero real $u\in C_c^\infty(D)$ whose Radon transform vanishes on every line meeting $E$, including its tangent lines. It is centrally symmetric about the center of $E$ and strictly negative on an open neighborhood of $E$; in particular, $E\subseteq\supp u$. Such witnesses can be chosen to form a countably infinite linearly independent family.
\end{proposition}

\begin{proof}
Write $E=c+Q^{1/2}\overline{\mathbb D_1}$ with $Q$ positive definite. Since $E\Subset D$, there is $r_{\mathrm{out}}>1$ with
\[
 c+Q^{1/2}\overline{\mathbb D_{r_{\mathrm{out}}}}\Subset D.
\]
Choose $1<r_{\mathrm{in}}<r_{\mathrm{out}}$, apply Lemma~\ref{lem:disks}, and define
\[
 u(x)=f\bigl(Q^{-1/2}(x-c)\bigr).
\]
Then $u$ is nonzero and real. By the sign conclusion in Lemma~\ref{lem:disks}, $u$ is strictly negative on an open neighborhood of $E$, and hence $E\subseteq\supp u$. Moreover,
\[
 \supp u\Subset c+Q^{1/2}\mathbb D_{r_{\mathrm{out}}}\Subset D.
\]
Because $f$ is radial, $u(c+x)=u(c-x)$.

Under $x=c+Q^{1/2}y$, the transformed line has unit normal $Q^{1/2}\omega/|Q^{1/2}\omega|$, offset $(p-c\cdot\omega)/|Q^{1/2}\omega|$, and line-element factor $\det Q^{1/2}/|Q^{1/2}\omega|$; here we have used that $Q^{1/2}$ is symmetric and $\det Q^{1/2}>0$. Thus
\begin{equation}\label{eq:affine-covariance}
 \Rad u(\omega,p)=
 \frac{\det Q^{1/2}}{|Q^{1/2}\omega|}
 \Rad f\left(\frac{Q^{1/2}\omega}{|Q^{1/2}\omega|},
          \frac{p-c\cdot\omega}{|Q^{1/2}\omega|}\right).
\end{equation}
Every line meeting $E$ satisfies
\[
 \frac{|p-c\cdot\omega|}{|Q^{1/2}\omega|}\le1<r_{\mathrm{in}}.
\]
Formula~\eqref{eq:affine-covariance} and Lemma~\ref{lem:disks} give the required vanishing.

For independence, choose $1<r_1<R_1<r_2<R_2<\cdots<r_{\mathrm{out}}$ and repeat the construction with disk radii $r_n,R_n$, obtaining $u_n$. At $x_n=c+Q^{1/2}(r_n,0)$, one has $u_n(x_n)<0$, whereas $u_m(x_n)=0$ for $m<n$. Evaluating any finite linear relation at the point corresponding to its largest index, and then descending, shows that all coefficients vanish.
\end{proof}

Together with Theorem~\ref{thm:main}, this shows that, for open convex sets $D_0,D$, the real smooth nullspace is either zero or infinite-dimensional.

\section{The symmetry-free transfer to unoriented-line space}

Identify $\R^2$ with $\C$ by $(x_1,x_2)\leftrightarrow x_1+ix_2$, and write
\[
 \omega_\theta=(\cos\theta,\sin\theta).
\]
The coordinates
\begin{equation*}
 q=e^{2i\theta},\qquad z=pe^{i\theta}
\end{equation*}
are invariant under $(\theta,p)\mapsto(\theta+\pi,-p)$ and hence parametrize unoriented lines. Their image is the set
\begin{equation*}
 \Sigma=\{(q,z)\in\T\times\C:z=q\overline z\},
\end{equation*}
which is the total space of the M\"obius real line bundle over $\T$. More explicitly, the map
\[
 (\omega,p)\longmapsto(\omega^2,p\omega),
 \qquad (\omega,p)\in\mathbb S^1\times\R,
\]
induces a homeomorphism from the standard quotient
$ (\mathbb S^1\times\R)/((\omega,p)\sim(-\omega,-p)) $ onto $\Sigma$.
The bundle projection is $[\omega,p]\mapsto\omega^2$; reversing the choice of square root reverses the real fiber coordinate.

For a nonempty bounded open convex set $K_0$ and a nonempty compact convex set $K_1$, define
\begin{equation*}
 \begin{split}
 X(K_0,K_1)=\bigl\{&((\omega_1+i\omega_2)^2,
             p(\omega_1+i\omega_2)):\ \omega\in \mathbb S^1,\\
 &p\in[a_{K_1}(\omega),b_{K_1}(\omega)]
       \setminus(a_{K_0}(\omega),b_{K_0}(\omega))\bigr\}.
 \end{split}
\end{equation*}
This is a compact subset of $\Sigma$. For $\theta\in\R$, write
\begin{equation*}
 X_\theta=\{p\in\R:(e^{2i\theta},pe^{i\theta})\in X(K_0,K_1)\}.
\end{equation*}
Thus
\begin{equation*}
 X_\theta=[a_{K_1}(\omega_\theta),b_{K_1}(\omega_\theta)]
 \setminus(a_{K_0}(\omega_\theta),b_{K_0}(\omega_\theta)).
\end{equation*}

\begin{lemma}[Unoriented-line transfer]\label{lem:transfer}
Let $K_0$ and $K_1$ be as above. Let $h:\R\times\R\to\C$ be jointly measurable, and suppose
\begin{equation}\label{eq:line-symmetry}
 h(\theta+\pi,-p)=h(\theta,p)
\end{equation}
for almost every $(\theta,p)$. Assume that $0\ne h\in L^1([0,2\pi)\times\R)$ and
\begin{equation}\label{eq:transfer-support}
 h(\theta,p)=0\quad\text{unless }p\in X_\theta
\end{equation}
almost everywhere. Suppose also that, for every $j\ge0$, the moment
\[
 h_j(\theta)=\int_\R p^j h(\theta,p)\,dp
\]
agrees almost everywhere with a trigonometric polynomial containing only the frequencies $e^{ik\theta}$ with
\begin{equation*}
 |k|\le j,\qquad k\equiv j\pmod2.
\end{equation*}
Then there is a nonzero finite regular complex Borel measure $\eta$ on $X(K_0,K_1)$ such that
\begin{equation}\label{eq:annihilation}
 \int_{X(K_0,K_1)}z^jq^\ell\,d\eta=0,
 \qquad j\ge0,\quad \ell\ge1.
\end{equation}
\end{lemma}

\begin{proof}
Boundedness of $K_1$ and Fubini's theorem give absolute convergence of every moment for almost every $\theta$. Write $X=X(K_0,K_1)$, and let
\[
 \Gamma:[0,\pi)\times\R\longrightarrow\Sigma,
 \qquad
 \Gamma(\theta,p)=(e^{2i\theta},pe^{i\theta}).
\]
The map is continuous and injective on its domain: equality of the first coordinates forces equality of the angles in $[0,\pi)$, and then equality of the second coordinates forces equality of the offsets. Since the domain is a Borel subset of $\R^2$, the Lusin--Souslin theorem makes $\Gamma$ a Borel isomorphism onto its image. Equation~\eqref{eq:line-symmetry} shows that the restriction of $h(\theta,p)\,dp\,d\theta$ to $[0,\pi)\times\R$ is a nonzero finite complex measure. Let $\eta$ be its pushforward under $\Gamma$. Injectivity and the Borel-isomorphism property make $\eta$ nonzero, and \eqref{eq:transfer-support} concentrates its total variation on the compact set $X(K_0,K_1)$. It is therefore a finite regular Borel measure on that set.

For $j\ge0$ and $\ell\ge1$, the pushforward formula gives
\begin{align*}
 \int_X z^jq^\ell\,d\eta
 &=\int_0^\pi e^{i(j+2\ell)\theta}
       \left(\int_\R p^jh(\theta,p)\,dp\right)d\theta\\
 &=\int_0^\pi e^{i(j+2\ell)\theta}h_j(\theta)\,d\theta.
\end{align*}
Every frequency in the last integrand is $e^{im\theta}$ with
\[
 m=j+2\ell+k,
 \qquad |k|\le j,
 \qquad k\equiv j\pmod2.
\]
Thus $m$ is an even integer and $m\ge2\ell>0$. Its integral over $[0,\pi)$ is zero, proving \eqref{eq:annihilation}.
\end{proof}

\section{From unoriented-line space to a Hardy-space model}

\begin{lemma}[Hardy-model reduction]\label{lem:hardy-model}
Let $X\subseteq\Sigma$ be compact. Suppose that a nonzero finite regular complex Borel measure $\eta$ on $X$ satisfies
\begin{equation}\label{eq:abstract-annihilation}
 \int_Xz^jq^\ell\,d\eta=0,
 \qquad j\ge0,
 \quad \ell\ge1.
\end{equation}
Then there are a nonzero complex Hilbert space $\mathcal E$ and an operator $A\in\mathcal B(\mathcal E)$ such that, for every $\theta\in\R$,
\begin{equation}\label{eq:fiber-spectrum}
 \sigma\bigl(e^{-i\theta}A+e^{i\theta}A^*\bigr)
 \subseteq\{p\in\R:(e^{2i\theta},pe^{i\theta})\in X\}.
\end{equation}
\end{lemma}

\begin{proof}
Let
\[
 \mathcal A=\overline{\C[q,z]}^{\,\|\cdot\|_X}
 \subseteq C(X).
\]
The coordinate $q$ is not invertible in $\mathcal A$. Indeed, if $q^{-1}\in\mathcal A$, then on $X\subseteq\Sigma$,
\[
 \overline q=q^{-1}\in\mathcal A,
 \qquad
 \overline z=q^{-1}z\in\mathcal A.
\]
Hence $\mathcal A$ would be self-adjoint. Since $q$ and $z$ separate the points of $X$, Stone--Weierstrass would give $\mathcal A=C(X)$. Moreover, $q\C[q,z]$ would be dense in $\mathcal A$, because one may approximate $q^{-1}f$ and then multiply by $q$. Equation~\eqref{eq:abstract-annihilation} would then make $\eta$ annihilate all of $C(X)$, contradicting $\eta\ne0$.

Choose a maximal ideal containing the proper ideal $q\mathcal A$. As usual, maximal ideals in a unital Banach algebra are closed, and the Gelfand--Mazur theorem gives a continuous unital character
\[
 \chi:\mathcal A\longrightarrow\C,
 \qquad \chi(q)=0.
\]
A norm-preserving Hahn--Banach extension $\Lambda$ of $\chi$ to $C(X)$ satisfies $\|\Lambda\|=1$ and $\Lambda(1)=1$. It is positive on real nonnegative functions. Indeed, for real $f\in C(X)$ and $s\in\R$,
\[
 |1+is\Lambda(f)|^2\le1+s^2\|f\|_X^2.
\]
Expanding this inequality for both signs of arbitrarily small $s$ gives $\Imag\Lambda(f)=0$. If $f\ge0$ and $r=\|f\|_X$, then $\|r-f\|_X\le r$, whence $r-\Real\Lambda(f)\le r$ and $\Lambda(f)\ge0$. The Riesz representation theorem for this positive functional gives a probability measure $\mu$ representing $\chi$:
\[
 \chi(f)=\int_Xf\,d\mu\qquad(f\in\mathcal A).
\]

Let $\nu$ be the $q$-marginal of $\mu$. For $n\ge1$,
\[
 \int_\T q^n\,d\nu=\chi(q^n)=0.
\]
Positivity gives the same conclusion for the negative Fourier coefficients. Hence $\nu$ is normalized Haar measure $m_\T$.

Let $M_q$ and $M_z$ denote the operators of multiplication by the coordinate functions $q$ and $z$ on $L^2(\mu)$, and put
\[
 \mathcal H=\overline{\mathcal A}^{\,L^2(\mu)},
 \qquad
 T=M_q|_{\mathcal H},
 \qquad
 S=M_z|_{\mathcal H}.
\]
Then $T$ is an isometry and $S$ commutes with $T$. For $f,g\in\mathcal H$,
\[
 \langle T^*Sf,g\rangle
 =\int_Xzf\,\overline{qg}\,d\mu
 =\int_X\overline z\,f\overline g\,d\mu
 =\langle S^*f,g\rangle,
\]
so $T^*S=S^*$. Taking adjoints gives $S^*T=S$; hence
\begin{equation}\label{eq:operator-relation}
 T^*S=S^*,
 \qquad
 S^*T=S.
\end{equation}

The constant vector $1$ has norm one and belongs to $\ker T^*$: for every polynomial $f\in\C[q,z]$, the representing identity gives $\int_Xqf\,d\mu=\chi(qf)=0$, and polynomial vectors are dense in $\mathcal H$. Consider the linear subspace
\[
 \mathcal L=\{f\in\mathcal H:T^{*n}f=0\text{ for some }n\ge0\}.
\]
The relations $T^*T=I$, $T^*S=S^*$, and $T^*S^*=S^*T^*$ show that $\mathcal L$ is invariant under both $T$ and $S$. It contains $1$, so it contains every polynomial vector $T^mS^j1=q^mz^j$. Thus every polynomial vector is killed by a sufficiently high power of $T^*$. Since $\|T^*\|\le1$, if $g$ is a polynomial vector and $T^{*N}g=0$, then
\[
 \|T^{*n}f\|\le\|f-g\|
 \qquad(n\ge N).
\]
Density of the polynomial vectors gives $T^{*n}f\to0$ for every $f\in\mathcal H$; thus $T$ is a \emph{pure} isometry.

Put $\mathcal E=\ker T^*$. This space is nonzero because it contains the unit vector $1$. The Wold decomposition identifies
\[
 \mathcal H=\bigoplus_{n\ge0}T^n\mathcal E
 \simeq\ell^2(\mathbb N,\mathcal E),
 \qquad T\simeq M_q.
\]
Here $\ell^2(\mathbb N,\mathcal E)$ is the coefficient realization of $H^2(\T,\mathcal E)$, and $M_q$ denotes its unilateral shift. Let $P_{\mathcal E}=I-TT^*$ and define $A=P_{\mathcal E}S|_{\mathcal E}$. Since $S^*$ commutes with $T^*$, it preserves $\mathcal E$, and $A^*=S^*|_{\mathcal E}$. Consequently, for $e\in\mathcal E$,
\[
 Se=Ae+TS^*e=Ae+TA^*e.
\]
Commutation of $S$ and $T$ now gives
\[
 ST^ne=T^nAe+T^{n+1}A^*e\qquad(n\ge0).
\]
Thus in the coefficient model,
\begin{equation}\label{eq:hardy-symbol}
 T\simeq M_q,
 \qquad
 S\simeq M_{A+qA^*}.
\end{equation}

We retain the fiberwise support information using finite sums in $\mathcal H$. Fix $\theta$, put $q_0=e^{2i\theta}$, and let
\[
 p_0\in\sigma(e^{-i\theta}A+e^{i\theta}A^*),
 \qquad z_0=e^{i\theta}p_0.
\]
Then $z_0\in\sigma(A+q_0A^*)$. Since
\[
 A+q_0A^*
 =e^{i\theta}(e^{-i\theta}A+e^{i\theta}A^*)
\]
is a unimodular scalar multiple of a self-adjoint operator, it is normal. If a normal operator minus a spectral value were bounded below, its range would be closed; normality would make the range dense, hence surjective, a contradiction. Thus every spectral point is an approximate eigenvalue, so there are unit vectors $e_n\in\mathcal E$ such that
\[
 \|(A+q_0A^*-z_0)e_n\|\longrightarrow0.
\]
For $e\in\mathcal E$ define
\[
 W_N(e)=\sum_{k=0}^N\overline q_0^{\,k}T^ke.
\]
Orthogonality of the Wold summands gives $\|W_N(e)\|=\sqrt{N+1}\,\|e\|$. The shift defect telescopes:
\[
 (T-q_0I)W_N(e)=\overline q_0^{\,N}T^{N+1}e-q_0e,
 \qquad
 \|(T-q_0I)W_N(e)\|\le2\|e\|.
\]
The formula for $ST^ne$ also gives
\[
 (S-z_0I)W_N(e)
 =W_N((A+q_0A^*-z_0I)e)+(T-q_0I)W_N(A^*e).
\]
Set $f_n=(n+1)^{-1/2}W_n(e_n)$. These are unit vectors in $\mathcal H$, and the preceding estimates yield
\[
 \|(T-q_0I)f_n\|\le\frac2{\sqrt{n+1}}\longrightarrow0,
 \qquad
 \|(S-z_0I)f_n\|
 \le\|(A+q_0A^*-z_0I)e_n\|
       +\frac{2\|A^*\|}{\sqrt{n+1}}\longrightarrow0.
\]
If $(q_0,z_0)\notin X$, compactness gives $\delta>0$ such that $|q-q_0|^2+|z-z_0|^2\ge\delta$ on $X$. Since $T$ and $S$ are restrictions of the coordinate multipliers on $L^2(\mu)$, every $f\in\mathcal H$ then satisfies
\[
 \|(T-q_0I)f\|^2+\|(S-z_0I)f\|^2\ge\delta\,\|f\|^2,
\]
contradicting the preceding approximate eigenvectors. Hence $(q_0,z_0)\in X$, which is exactly \eqref{eq:fiber-spectrum}.
\end{proof}

\section{A spectral gap produces an ellipse}

No finite-dimensionality is assumed before the final $2\times2$ construction. The functions $a_1,a_0,b_0,b_1$ appearing in Theorem~\ref{thm:operator-ellipse} below are abstract data; in the application (Proposition~\ref{prop:ellipse}) they will be the support-function endpoints \eqref{eq:endpoints} of the two convex sets, evaluated at $\omega_\theta$.

\begin{lemma}[A common neutral subspace]\label{lem:neutral}
Let $H_0,H_1$ be bounded self-adjoint operators on a nonzero complex Hilbert space $\mathcal H$. Suppose that
\begin{equation}\label{eq:pencil-invertible}
 H_0\cos\theta+H_1\sin\theta
\end{equation}
is invertible for every $\theta\in\R$. Then there is a nonzero closed subspace $\mathcal M\subset\mathcal H$, with orthogonal projection $P_{\mathcal M}$ onto it, such that, for every $\theta$,
\begin{equation}\label{eq:neutral-conclusion}
 P_{\mathcal M}(H_0\cos\theta+H_1\sin\theta)|_{\mathcal M}=0,
 \qquad
 (H_0\cos\theta+H_1\sin\theta)\mathcal M=\mathcal M^\perp.
\end{equation}
The first identity says that $\mathcal M$ is \emph{neutral} for the pencil, that is, $\langle(H_0\cos\theta+H_1\sin\theta)e,e'\rangle=0$ for all $e,e'\in\mathcal M$ and all $\theta$.
\end{lemma}

\begin{proof}
The operator $H_0$ is invertible. Put
\[
 R=H_0^{-1}H_1.
\]
Then
\begin{equation}\label{eq:R-selfadjoint}
 R^*H_0=H_0R,
 \qquad
 R^*=H_0RH_0^{-1}.
\end{equation}
The invertibility of \eqref{eq:pencil-invertible} implies
\[
 \sigma(R)\cap\R=\varnothing.
\]
Indeed, if $t\in\sigma(R)\cap\R$, choose $\theta$ with $\cos\theta+t\sin\theta=0$ and apply spectral mapping to $\cos\theta I+\sin\theta R$. Moreover, \eqref{eq:R-selfadjoint} and $\sigma(R^*)=\{\overline\lambda:\lambda\in\sigma(R)\}$ show that $\sigma(R)$ is invariant under complex conjugation.

Because $\sigma(R)$ is compact and disjoint from $\R$, its intersections with the upper and lower half-planes are separated compact spectral sets. Let $P_+$ and $P_-$ be the corresponding Riesz projections. Since $\sigma(R)$ is nonempty and invariant under complex conjugation, both projections are nonzero, and
\[
 \mathcal H=\ran P_+\dotplus\ran P_-
 \qquad\text{(topological direct sum)}.
\]
Taking adjoints in the Riesz-integral formula shows that $P_+^*$ is the Riesz projection of $R^*$ associated with the lower half-plane. The similarity in \eqref{eq:R-selfadjoint} therefore gives
\[
 P_+^*=H_0P_-H_0^{-1},
\]
or equivalently
\begin{equation}\label{eq:riesz-adjoint}
 P_+^*H_0=H_0P_-.
\end{equation}
Set $\mathcal M=\ran P_+$ and $W=\ran P_-$. Since $P_-P_+=0$, equation~\eqref{eq:riesz-adjoint} gives
\[
 P_+^*H_0P_+=0.
\]
Because $H_1=H_0R$ and $R$ commutes with $P_+$, one also has
\[
 P_+^*H_1P_+=0.
\]
Thus $\langle H_je,e'\rangle=0$ for $e,e'\in\mathcal M$ and $j=0,1$, equivalently $H_j\mathcal M\subseteq\mathcal M^\perp$. The same argument shows that $W$ is neutral for both $H_0$ and $H_1$.

Fix $G=H_0\cos\theta+H_1\sin\theta$. Neutrality gives $G\mathcal M\subseteq\mathcal M^\perp$. Conversely, let $x=e+w\in G^{-1}(\mathcal M^\perp)$, with $e\in\mathcal M$ and $w\in W$. For every $e'\in\mathcal M$,
\[
 \langle Gw,e'\rangle
 =\langle Gx,e'\rangle-\langle Ge,e'\rangle=0.
\]
Neutrality of $W$ also gives $\langle Gw,w'\rangle=0$ for every $w'\in W$. Since $\mathcal M\dotplus W=\mathcal H$, it follows that $Gw=0$. The operator $G$ is invertible, so $w=0$ and $x\in\mathcal M$. Hence $G^{-1}(\mathcal M^\perp)=\mathcal M$, and therefore $G\mathcal M=\mathcal M^\perp$.
\end{proof}

\begin{lemma}[The ellipse of a $2\times2$ matrix]\label{lem:matrix-ellipse}
For every matrix $C\in \C^{2\times2}$ there are $c\in\R^2$ and a real positive semidefinite matrix $Q$ such that, for every $\theta$, the two eigenvalues of the Hermitian matrix $\Real(e^{-i\theta}C)$ are
\begin{equation}\label{eq:matrix-eigen-endpoints}
 c\cdot\omega_\theta
 \mathbin{-}(\omega_\theta^{\mathsf T}Q\omega_\theta)^{1/2},
 \qquad
 c\cdot\omega_\theta
 \mathbin{+}(\omega_\theta^{\mathsf T}Q\omega_\theta)^{1/2}.
\end{equation}
Consequently these eigenvalues are the projection endpoints of the possibly degenerate elliptic disk
\[
 E_C=c+Q^{1/2}\overline{\mathbb D_1}.
\]
If the two eigenvalues are distinct for every $\theta$, then $Q$ is positive definite and $E_C$ is an ellipse.
\end{lemma}

\begin{proof}
Put
\[
 \tau=\frac12\tr C,
 \qquad
 c=(\Real\tau,\Imag\tau),
\]
and define the traceless Hermitian matrices
\[
 C_1=\Real(C-\tau I),
 \qquad
 C_2=\Imag(C-\tau I).
\]
Then
\[
 \Real(e^{-i\theta}C)
 =(c\cdot\omega_\theta)I+C_1\cos\theta+C_2\sin\theta.
\]
A traceless $2\times2$ Hermitian matrix $\Xi$ has eigenvalues
\[
 \pm\left(\frac12\tr(\Xi^2)\right)^{1/2}.
\]
Set
\begin{equation*}
 Q=
 \begin{pmatrix}
 \frac12\tr(C_1^2)&\frac12\tr(C_1C_2)\\
 \frac12\tr(C_1C_2)&\frac12\tr(C_2^2)
 \end{pmatrix}.
\end{equation*}
For $\omega_\theta=(\cos\theta,\sin\theta)$,
\[
 \omega_\theta^{\mathsf T}Q\omega_\theta
 =\frac12\tr\bigl((C_1\cos\theta+C_2\sin\theta)^2\bigr)\ge0.
\]
This proves \eqref{eq:matrix-eigen-endpoints}. If the eigenvalues are distinct for every $\theta$, the last quadratic form is strictly positive on $\mathbb S^1$, so $Q$ is positive definite.
\end{proof}

\begin{theorem}[Spectral gap]\label{thm:operator-ellipse}
Let $\mathbf A$ be a bounded operator on a nonzero complex Hilbert space $\mathcal H$. Let $a_1,a_0,b_0,b_1:\R\to\R$ be arbitrary functions, and suppose
\begin{equation}\label{eq:operator-band}
 \sigma\bigl(\Real(e^{-i\theta}\mathbf A)\bigr)
 \subseteq[a_1(\theta),b_1(\theta)]
       \setminus(a_0(\theta),b_0(\theta))
\end{equation}
for every $\theta$. Suppose also that some $z_*\in\C$ satisfies
\begin{equation}\label{eq:gap-point}
 a_0(\theta)<\Real(e^{-i\theta}z_*)<b_0(\theta)
 \qquad(\theta\in\R).
\end{equation}
Then there is a matrix $C\in \C^{2\times2}$ such that, if
\[
 \lambda_-(\theta)<\lambda_+(\theta)
\]
are the two eigenvalues of $\Real(e^{-i\theta}C)$, then
\begin{equation}\label{eq:matrix-band-sandwich}
 a_1(\theta)\le\lambda_-(\theta)\le a_0(\theta)
 <b_0(\theta)\le\lambda_+(\theta)\le b_1(\theta)
\end{equation}
for every $\theta$.
\end{theorem}

\begin{proof}
Set
\[
 G_0=\Real(\mathbf A-z_*I),
 \qquad
 G_1=\Imag(\mathbf A-z_*I).
\]
By \eqref{eq:operator-band} and \eqref{eq:gap-point}, the self-adjoint pencil
\[
 G_0\cos\theta+G_1\sin\theta
 =\Real(e^{-i\theta}(\mathbf A-z_*I))
\]
is invertible for every $\theta$. Apply Lemma~\ref{lem:neutral}. Its neutrality conclusion at $\theta=0$ and $\theta=\pi/2$ gives
\[
 P_{\mathcal M}(\mathbf A-z_*I)|_{\mathcal M}=0.
\]
Relative to the orthogonal decomposition $\mathcal H=\mathcal M\oplus\mathcal N$, where $\mathcal N=\mathcal M^\perp$, the operator $\mathbf A$ therefore has the block form
\begin{equation*}
 \mathbf A=
 \begin{pmatrix}
 z_*I&\mathbf A_{12}\\ \mathbf A_{21}&\mathbf A_{22}
 \end{pmatrix},
 \qquad
 \mathbf A_{12}\in\mathcal B(\mathcal N,\mathcal M),\quad
 \mathbf A_{21}\in\mathcal B(\mathcal M,\mathcal N),\quad
 \mathbf A_{22}\in\mathcal B(\mathcal N).
\end{equation*}
For
\[
 H_\theta=\Real(e^{-i\theta}\mathbf A),
 \qquad
 \alpha_\theta=\Real(e^{-i\theta}z_*),
\]
write
\begin{equation}\label{eq:H-block}
 H_\theta=
 \begin{pmatrix}
 \alpha_\theta I&Y_\theta\\
 Y_\theta^*&\mathcal D_\theta
 \end{pmatrix},
 \qquad
 Y_\theta=\frac12(e^{-i\theta}\mathbf A_{12}+e^{i\theta}\mathbf A_{21}^*)\in\mathcal B(\mathcal N,\mathcal M),
 \qquad
 \mathcal D_\theta=\Real(e^{-i\theta}\mathbf A_{22}).
\end{equation}
For $m\in\mathcal M$, the block form gives
\[
 (H_\theta-\alpha_\theta I)(m,0)=(0,Y_\theta^*m).
\]
The second conclusion of Lemma~\ref{lem:neutral}, applied to $H_\theta-\alpha_\theta I$, says that the left-hand side ranges over $\mathcal N$, so $Y_\theta^*:\mathcal M\to\mathcal N$ is onto. The open mapping theorem and duality imply that its adjoint is bounded below: for some $\kappa_\theta>0$,
\[
 \|Y_\theta w\|\ge\kappa_\theta\|w\|\qquad(w\in\mathcal N).
\]
Moreover, $\mathcal N\ne\{0\}$: the space $\mathcal M$ is nonzero and the invertible operator $H_\theta-\alpha_\theta I$ maps $\mathcal M$ onto $\mathcal N$.

Choose a unit vector $w\in\mathcal N$ and set
\[
 d=\langle\mathbf A_{22}w,w\rangle,
 \qquad
 r=\|\mathbf A_{12}w\|^2,
 \qquad
 s=\|\mathbf A_{21}^*w\|^2,
 \qquad
 \zeta=\langle\mathbf A_{12}w,\mathbf A_{21}^*w\rangle.
\]
By Cauchy--Schwarz,
\[
 2|\zeta|\le r+s.
\]
Therefore one can choose $\beta,\gamma\in\C$ such that
\begin{equation}\label{eq:beta-gamma}
 |\beta|^2+|\gamma|^2=r+s,
 \qquad
 \beta\gamma=\zeta.
\end{equation}
For instance, choose nonnegative numbers $\rho_1,\rho_2$ with $\rho_1+\rho_2=r+s$ and $\rho_1\rho_2=|\zeta|^2$, and then choose phases so that $\beta\gamma=\zeta$. Define the auxiliary matrix
\begin{equation*}
 C=
 \begin{pmatrix}
 z_*&\beta\\
 \gamma&d
 \end{pmatrix}.
\end{equation*}
The off-diagonal entry $x_\theta$ of $\Real(e^{-i\theta}C)$ is
\[
 x_\theta=\frac12(e^{-i\theta}\beta+e^{i\theta}\overline\gamma).
\]
Equations~\eqref{eq:H-block} and \eqref{eq:beta-gamma} give the identity
\begin{equation}\label{eq:offdiag-identity}
 |x_\theta|^2=\|Y_\theta w\|^2
 \qquad(\theta\in\R).
\end{equation}
In particular, $x_\theta\ne0$ for every $\theta$, so the two eigenvalues of $\Real(e^{-i\theta}C)$ are distinct and strictly straddle $\alpha_\theta$.

Fix $\theta$ and abbreviate $H=H_\theta$, $\alpha=\alpha_\theta$, $Y=Y_\theta$, and $\mathcal D=\mathcal D_\theta$. For $t\ne\alpha$, define the self-adjoint Schur complement on $\mathcal N$ by
\begin{equation*}
 \mathcal S(t)=\mathcal D-tI_{\mathcal N}-\frac{Y^*Y}{\alpha-t}.
\end{equation*}
Here $I_{\mathcal N}$ denotes the identity operator on $\mathcal N$, and $\mathcal S(t)>0$ and $\mathcal S(t)<0$ mean, respectively, $\mathcal S(t)\ge\varepsilon I_{\mathcal N}$ and $\mathcal S(t)\le-\varepsilon I_{\mathcal N}$ for some $\varepsilon>0$. Completing the square gives, for $w\in\mathcal N$ and $v_t=(-Yw/(\alpha-t),w)$,
\begin{equation}\label{eq:schur-completing-square}
 (H-tI)v_t=(0,\mathcal S(t)w),
 \qquad
 \langle\mathcal S(t)w,w\rangle=\langle(H-tI)v_t,v_t\rangle.
\end{equation}
Solving the first block equation for the $\mathcal M$-component also shows that $H-tI$ is invertible if and only if $\mathcal S(t)$ is invertible.

The lower bound for $Y$ gives $Y^*Y\ge\kappa_\theta^2 I_{\mathcal N}$. Since $\mathcal D-tI_{\mathcal N}$ remains bounded as $t\to\alpha$, it follows that $\mathcal S(t)<0$ for $t<\alpha$ sufficiently close to $\alpha$, and $\mathcal S(t)>0$ for $t>\alpha$ sufficiently close to $\alpha$. The interval $(a_0(\theta),b_0(\theta))$ contains no spectrum of $H$, so $\mathcal S(t)$ is a norm-continuous path of invertible self-adjoint operators on each of $(a_0(\theta),\alpha)$ and $(\alpha,b_0(\theta))$. For such a path, the set of parameters where the operator is nonnegative is closed, because the positive cone is norm-closed. It is also open: a nonnegative invertible self-adjoint operator is bounded below by $\varepsilon I$ for some $\varepsilon>0$, and a sufficiently small norm perturbation preserves positivity. Connectedness therefore propagates positivity from any one parameter to the whole interval. Applying this argument to $\mathcal S$ and $-\mathcal S$ gives
\begin{equation}\label{eq:schur-gap-sign}
 \mathcal S(t)<0\quad(a_0(\theta)<t<\alpha),
 \qquad
 \mathcal S(t)>0\quad(\alpha<t<b_0(\theta)).
\end{equation}
The outer spectral bounds give $H-tI\ge0$ for $t<a_1(\theta)$ and $tI-H\ge0$ for $t>b_1(\theta)$. Equation~\eqref{eq:schur-completing-square} therefore gives
\begin{equation}\label{eq:schur-outer-sign}
 \mathcal S(t)\ge0\quad(t<\min\{a_1(\theta),\alpha\}),
 \qquad
 \mathcal S(t)\le0\quad(t>\max\{b_1(\theta),\alpha\}).
\end{equation}

Let $\delta_\theta=\Real(e^{-i\theta}d)=\langle\mathcal Dw,w\rangle$. For $t\ne\alpha$, \eqref{eq:offdiag-identity} and the characteristic polynomial of $\Real(e^{-i\theta}C)$ give
\begin{equation}\label{eq:scalar-schur}
 \langle\mathcal S(t)w,w\rangle
 =\delta_\theta-t-\frac{|x_\theta|^2}{\alpha-t}
 =\frac{(\lambda_-(\theta)-t)(\lambda_+(\theta)-t)}{\alpha-t}.
\end{equation}
Since $x_\theta\ne0$, one has $\lambda_-(\theta)<\alpha<\lambda_+(\theta)$. Thus the scalar in \eqref{eq:scalar-schur} is positive on $(-\infty,\lambda_-(\theta))$, negative on $(\lambda_-(\theta),\alpha)$, positive on $(\alpha,\lambda_+(\theta))$, and negative on $(\lambda_+(\theta),\infty)$. Comparing this sign pattern with \eqref{eq:schur-gap-sign} and \eqref{eq:schur-outer-sign} yields
\[
 a_1(\theta)\le\lambda_-(\theta)\le a_0(\theta),
 \qquad
 b_0(\theta)\le\lambda_+(\theta)\le b_1(\theta).
\]
For example, if $\lambda_-(\theta)<a_1(\theta)$, choose $t$ strictly between $\lambda_-(\theta)$ and $\min\{a_1(\theta),\alpha\}$; the scalar is negative while \eqref{eq:schur-outer-sign} makes it nonnegative. The other three inequalities follow in the same way. This proves \eqref{eq:matrix-band-sandwich}.
\end{proof}

\section{Necessity and the main theorem}

\begin{proposition}[Necessary ellipse]\label{prop:ellipse}
Let $K_0\subseteq\R^2$ be convex with nonempty interior, and let $K_1\subseteq\R^2$ be compact and convex. Suppose that $0\ne v\in L^1(\R^2)$, $\supp v\subseteq K_1$, and
\begin{equation}\label{eq:necessary-vanishing}
 \Rad v(\omega,p)=0
 \quad\text{for almost every }(\omega,p)\text{ with }
 L(\omega,p)\cap K_0\ne\varnothing.
\end{equation}
Then there is an ellipse $E$ such that
\begin{equation*}
 \overline{K_0}\subseteq E\subseteq K_1.
\end{equation*}
\end{proposition}

\begin{proof}
Replacing $K_0$ by its interior preserves its closure, so assume $K_0$ is open. First suppose it is bounded; $K_1$ is nonempty since $v\ne0$. Choose a Borel representative of $v$, equal to the original class almost everywhere, that vanishes pointwise outside $K_1$. Fubini's theorem in line coordinates shows that this replacement does not change the joint Radon $L^1$-class. Define
\begin{equation*}
 h(\theta,p)=\int_\R
 v(p\omega_\theta+t\omega_\theta^\perp)\,dt,
\end{equation*}
setting the value to zero where the integral is not absolutely convergent. Tonelli's theorem shows that $h$ is jointly measurable and
\[
 \int_\R|h(\theta,p)|\,dp\le\|v\|_{L^1(\R^2)}.
\]
By the parametrization $x=p\omega_\theta+t\omega_\theta^\perp$ of $L(\omega_\theta,p)$, one has $h(\theta,p)=\Rad v(\omega_\theta,p)$ for almost every $(\theta,p)$. The function $h$ satisfies
\begin{equation}\label{eq:radon-line-symmetry}
 h(\theta+\pi,-p)=h(\theta,p)
\end{equation}
for almost every $(\theta,p)$.

The function $h$ is nonzero in $L^1([0,2\pi)\times\R)$. With the convention
\[
 \widehat v(\xi)=\int_{\R^2}e^{-ix\cdot\xi}v(x)\,dx,
\]
Fubini's theorem gives the Fourier-slice identity
\[
 \int_\R e^{-isp}h(\theta,p)\,dp
 =\widehat v(s\omega_\theta).
\]
If $h=0$ almost everywhere, then $\widehat v$ vanishes on every radial line in a full-measure, hence dense, set of directions. Continuity gives $\widehat v\equiv0$, and injectivity of the Fourier transform on $L^1(\R^2)$ gives $v=0$, a contradiction.

In the notation of \eqref{eq:endpoints}, write
\[
 a_j(\theta)=a_{K_j}(\omega_\theta),
 \qquad
 b_j(\theta)=b_{K_j}(\omega_\theta),
 \qquad j=0,1.
\]
A line meets $K_0$ exactly when $a_0(\theta)<p<b_0(\theta)$, while a line with $p\notin[a_1(\theta),b_1(\theta)]$ misses $K_1$. Hence \eqref{eq:necessary-vanishing} and the support of $v$ give
\begin{equation*}
 h(\theta,p)=0
 \quad\text{unless}\quad
 p\in[a_1(\theta),b_1(\theta)]\setminus(a_0(\theta),b_0(\theta))
\end{equation*}
almost everywhere.

For every $j\ge0$, boundedness of $K_1$ gives absolute convergence of the moments, and Fubini's theorem yields
\begin{equation}\label{eq:radon-moment-general}
 h_j(\theta):=\int_\R p^jh(\theta,p)\,dp
 =\int_{\R^2}(x\cdot\omega_\theta)^jv(x)\,dx.
\end{equation}
Expanding $x\cdot\omega_\theta$ in powers of $e^{i\theta}$ and $e^{-i\theta}$ shows that $h_j$ contains only frequencies $e^{ik\theta}$ with $|k|\le j$ and $k\equiv j\pmod2$. The line-symmetry relation \eqref{eq:radon-line-symmetry}, the preceding support restriction, and the moment identity \eqref{eq:radon-moment-general} verify the hypotheses of Lemma~\ref{lem:transfer}, which gives a nonzero measure $\eta$ on $X(K_0,K_1)$ satisfying \eqref{eq:annihilation}.

Apply Lemma~\ref{lem:hardy-model}. There are a nonzero complex Hilbert space $\mathcal E$ and $A\in\mathcal B(\mathcal E)$ such that
\begin{equation}\label{eq:model-band-raw}
 \sigma(e^{-i\theta}A+e^{i\theta}A^*)
 \subseteq[a_1(\theta),b_1(\theta)]
       \setminus(a_0(\theta),b_0(\theta)).
\end{equation}
Put $\mathbf A=2A$. Then the left-hand side of \eqref{eq:model-band-raw} is
\[
 \sigma(\Real(e^{-i\theta}\mathbf A)).
\]
Choose any $x_*\in K_0$ and identify it with $z_*=x_{*,1}+ix_{*,2}\in\C$. Since $K_0$ is open,
\begin{equation*}
 a_0(\theta)<x_*\cdot\omega_\theta
 =\Real(e^{-i\theta}z_*)<b_0(\theta)
\end{equation*}
for every $\theta$. Theorem~\ref{thm:operator-ellipse} therefore gives a matrix $C\in \C^{2\times2}$ such that the eigenvalues $\lambda_-(\theta)<\lambda_+(\theta)$ of $\Real(e^{-i\theta}C)$ satisfy
\begin{equation}\label{eq:final-projection-band}
 a_1(\theta)\le\lambda_-(\theta)\le a_0(\theta)
 <b_0(\theta)\le\lambda_+(\theta)\le b_1(\theta).
\end{equation}
By Lemma~\ref{lem:matrix-ellipse}, these two functions are the projection endpoints of an ellipse $E$. Equation~\eqref{eq:final-projection-band} and the support-function criterion \eqref{eq:support-inclusion} give
\[
 \overline{K_0}\subseteq E\subseteq K_1.
\]
For general $K_0$, choose for each $x\in K_0$ an open disk $U_x$ with $x\in U_x\subseteq K_0$. The bounded case applied to $U_x$ gives $U_x\subseteq K_1$. Hence $K_0\subseteq K_1$ is bounded, and the bounded case applies to $K_0$ itself.
\end{proof}

\begin{proof}[Proof of Theorem~\ref{thm:main}]
Proposition~\ref{prop:sufficiency} shows that (iii) implies (i), with the stated reality, symmetry, and sign properties. Condition (i) implies (ii), because a compactly supported smooth function is integrable and its essential support equals its ordinary support.

Suppose (ii) holds, so $D\ne\varnothing$. If $D_0=\varnothing$, any sufficiently small closed disk in $D$ satisfies (iii). Otherwise, let
\[
 K_1=\co(\supp u).
\]
The set $K_1$ is nonempty, compact, and convex. Since $D$ is convex, $K_1\subseteq D$; compactness then gives $K_1\Subset D$.

Proposition~\ref{prop:ellipse}, applied with $K_0=D_0$ and $v=u$, produces an ellipse satisfying
\[
 \overline{D_0}\subseteq E\subseteq K_1\Subset D.
\]
Thus (ii) implies (iii). Since (iii) gives $D_0\Subset D$, so do (i) and (ii).
\end{proof}

\section{The square threshold and earlier claims}

\begin{corollary}[Two concentric squares]\label{cor:squares}
Let $\delta>0$ and set
\[
 D=(-1,1)^2,
 \qquad
 D_0=(-\delta,\delta)^2.
\]
There is a nonzero $u\in C_c^\infty(D)$ satisfying \eqref{eq:smooth-vanishing}, or equivalently a nonzero $u\in L^1(\R^2)$ with $\supp u\Subset D$ satisfying \eqref{eq:ae-vanishing}, if and only if
\[
 \delta<\frac1{\sqrt2}.
\]
\end{corollary}

\begin{proof}
If $\delta<1/\sqrt2$, choose $r$ with $\delta\sqrt2<r<1$. The closed disk $\overline{\mathbb D_r}$ contains $\overline{D_0}$ and is contained in $D$ (Figure~\ref{fig:squares}, left), so Theorem~\ref{thm:main} applies.

Conversely, suppose an ellipse
\[
 E=c+Q^{1/2}\overline{\mathbb D_1}
\]
satisfies $\overline{D_0}\subseteq E\Subset D$. Projection onto the two coordinate axes gives, with $Q_{ij}$ and $c_i$ the entries of $Q$ and $c$,
\begin{equation}\label{eq:square-outer-coordinate}
 \sqrt{Q_{ii}}<1-|c_i|,
 \qquad i=1,2.
\end{equation}
In particular, $\tr Q<2$. For
\[
 \omega_+=\frac{(1,1)}{\sqrt2},
 \qquad
 \omega_-=\frac{(1,-1)}{\sqrt2},
\]
comparison of the projection intervals of $[-\delta,\delta]^2$ and $E$ gives
\[
 (\omega_\pm^{\mathsf T}Q\omega_\pm)^{1/2}
 \ge\delta\sqrt2+|c\cdot\omega_\pm|
 \ge\delta\sqrt2.
\]
Since $\omega_+,\omega_-$ are an orthonormal basis,
\[
 \tr Q
 =\omega_+^{\mathsf T}Q\omega_++\omega_-^{\mathsf T}Q\omega_-
 \ge4\delta^2.
\]
Thus $\delta\ge1/\sqrt2$ would give $\tr Q\ge2$, contradicting \eqref{eq:square-outer-coordinate}. Theorem~\ref{thm:main} completes the proof.
\end{proof}

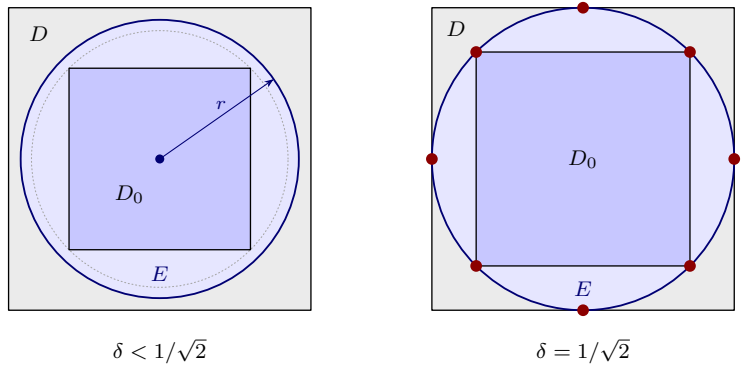
\begin{figure}[t]
\centering
\begin{tikzpicture}[line cap=round, line join=round, scale=2.0]
\begin{scope}
\fill[black!8] (-1,-1) rectangle (1,1);
\fill[blue!10] (0,0) circle (0.92);
\fill[blue!22] (-0.6,-0.6) rectangle (0.6,0.6);
\draw[gray!70, densely dotted, line width=0.4pt] (0,0) circle (0.8485);
\draw[black, line width=0.5pt] (-1,-1) rectangle (1,1);
\draw[blue!45!black, line width=0.7pt] (0,0) circle (0.92);
\draw[black, line width=0.5pt] (-0.6,-0.6) rectangle (0.6,0.6);
\draw[blue!45!black, line width=0.4pt, -{Stealth[length=3.5pt]}]
 (0,0) -- (35:0.92) node[pos=0.60, above left, inner sep=0.5pt] {\scriptsize $r$};
\fill[blue!45!black] (0,0) circle (0.85pt);
\node at (-0.80,0.83) {\footnotesize $D$};
\node at (-0.20,-0.23) {\footnotesize $D_0$};
\node[blue!45!black] at (0,-0.76) {\footnotesize $E$};
\node at (0,-1.27) {\footnotesize $\delta<1/\sqrt2$};
\end{scope}
\begin{scope}[shift={(2.8,0)}]
\fill[black!8] (-1,-1) rectangle (1,1);
\fill[blue!10] (0,0) circle (1);
\fill[blue!22] (-0.7071,-0.7071) rectangle (0.7071,0.7071);
\draw[black, line width=0.5pt] (-1,-1) rectangle (1,1);
\draw[blue!45!black, line width=0.7pt] (0,0) circle (1);
\draw[black, line width=0.5pt] (-0.7071,-0.7071) rectangle (0.7071,0.7071);
\foreach \p in {(1,0),(-1,0),(0,1),(0,-1),
                (0.7071,0.7071),(0.7071,-0.7071),
                (-0.7071,0.7071),(-0.7071,-0.7071)}
 \fill[red!55!black] \p circle (1.1pt);
\node at (-0.84,0.86) {\footnotesize $D$};
\node at (0,0) {\footnotesize $D_0$};
\node[blue!45!black] at (0,-0.855) {\footnotesize $E$};
\node at (0,-1.27) {\footnotesize $\delta=1/\sqrt2$};
\end{scope}
\end{tikzpicture}
\caption{The square threshold of Corollary~\ref{cor:squares}, with $D=(-1,1)^2$ and $D_0=(-\delta,\delta)^2$. Left: for $\delta<1/\sqrt2$, a disk $E=\overline{\mathbb D_r}$ with $\delta\sqrt2<r<1$ contains $\overline{D_0}$ and is compactly contained in $D$; the dotted circle has radius $\delta\sqrt2$ and passes through the corners of $\overline{D_0}$. Right: at $\delta=1/\sqrt2$, the closed unit disk is the unique ellipse satisfying $[-\delta,\delta]^2\subseteq E\subseteq[-1,1]^2$; it passes through the four corners of the inner square and meets $\partial D$ at the four side midpoints, so it is not compactly contained in $D$ (Remark~\ref{rem:endpoint}).}
\label{fig:squares}
\end{figure}

The uniqueness conclusion also holds with support in the closed outer square and data on lines meeting the closed inner square.

\begin{corollary}[Support in the closed outer square]\label{cor:closed}
Let $\delta>1/\sqrt2$. Then the only $u\in L^1(\R^2)$ with $\supp u\subseteq[-1,1]^2$ whose Radon transform vanishes almost everywhere on the set of lines meeting $[-\delta,\delta]^2$ is $u=0$.
\end{corollary}

\begin{proof}
Suppose $u\ne0$. The set
\[
 K_1=\co(\supp u)\subseteq[-1,1]^2
\]
is compact and convex. Proposition~\ref{prop:ellipse}, with $K_0=[-\delta,\delta]^2$, gives an ellipse
\[
 [-\delta,\delta]^2\subseteq E=c+Q^{1/2}\overline{\mathbb D_1}
 \subseteq[-1,1]^2.
\]
The outer inclusion gives
\[
 \sqrt{Q_{ii}}\le1-|c_i|,
 \qquad i=1,2,
\]
and hence $\tr Q\le2$. The two diagonal directions used in Corollary~\ref{cor:squares} give $\tr Q\ge4\delta^2>2$, a contradiction.
\end{proof}

\begin{remark}[The endpoint]\label{rem:endpoint}
At $\delta=1/\sqrt2$, the closed-support argument forces equality throughout. Indeed,
\[
 \tr Q\le(1-|c_1|)^2+(1-|c_2|)^2\le2,
\]
while the two diagonal directions give $\tr Q\ge2$. Hence $c=0$ and $Q_{11}=Q_{22}=1$. The inequalities
\[
 \omega_\pm^{\mathsf T}Q\omega_\pm=1\pm Q_{12}\ge1
\]
then force $Q_{12}=0$. Thus the closed unit disk is the unique ellipse satisfying
\[
 [-1/\sqrt2,1/\sqrt2]^2\subseteq E\subseteq[-1,1]^2.
\]
See Figure~\ref{fig:squares} (right). Proposition~\ref{prop:sufficiency} does not apply because the disk is not contained in the open square. We are not aware of a nonzero $u\in L^1(\R^2)$ with $\supp u\subseteq[-1,1]^2$ whose Radon transform vanishes almost everywhere on the set of lines meeting $[-1/\sqrt2,1/\sqrt2]^2$. By contrast, Theorem~\ref{thm:main} settles the compact-support-in-the-open-square problem at the endpoint in the negative.
\end{remark}

\begin{remark}[Relation to earlier claims]\label{rem:boman}
Conjecture~1.2 of \cite{BomanEllipsoidII} asserts that, for every bounded convex planar domain $D$ and every closed set $K\Subset D$, there is a nonzero smooth function $f$ with $\supp f\subseteq\overline D$ whose Radon transform vanishes on every line meeting $K$. Taking $D=(-1,1)^2$ and $K=[-\delta,\delta]^2$ with $1/\sqrt2<\delta<1$, Corollary~\ref{cor:closed} shows that the conjecture is false. Theorem~40.1 of \cite{BomanIPMS} asserts that condition~(i) of Theorem~\ref{thm:main} holds for every bounded pair $D_0\Subset D$ satisfying its remaining geometric hypotheses; taking $D_0=(-\delta,\delta)^2$ and the same $D$ gives a counterexample by Corollary~\ref{cor:squares}. In the proof of~(B) \cite[pp.~330--332]{BomanIPMS}, $b>0$ is fixed before $M_1=M_1(b)$ is chosen, and only then is a minimizer $u_0=u_b$ for $\mu_{M_1(b)}$ selected. The resulting bound by $\sqrt b$ therefore concerns a $b$-dependent family, and a bound tending to zero along such a family does not by itself yield a single function with exactly vanishing data. Our counterexamples do not depend on the limiting values of the variational minima $\mu_M$.
\end{remark}

\section{Further questions}

For general closed or disconnected inner sets, we are not aware of a geometric characterization beyond Remark~\ref{rem:connected}. If $D_0=\{x_*\}\subset D$, any nonzero smooth function $g$ supported in a sufficiently small ball about $x_*$ and satisfying $g(2x_*-x)=-g(x)$ has zero integral over every line through $x_*$.

A natural higher-dimensional question concerns the \emph{X-ray transform}, which integrates over affine lines. For $n\ge3$ and nonempty bounded open convex sets $D_0\Subset D\subset\R^n$, is the existence of $0\ne u\in C_c^\infty(D)$ with $\int_Lu\,ds=0$ for every affine line $L$ meeting $D_0$ equivalent to the existence of a closed ellipsoid $E$ with nonempty interior satisfying $\overline{D_0}\subseteq E\Subset D$? Here $ds$ is Euclidean arc length.

The hyperplane Radon transform has different inversion properties. For $n\ge2$, write
\[
 \Rad_nu(\omega,p)
 =\int_{\{x\in\R^n:x\cdot\omega=p\}}u(x)\,d\lambda_{\omega,p}(x),
 \qquad
 \Rad_n^*\varphi(x)=\int_{\mathbb S^{n-1}}\varphi(\omega,x\cdot\omega)\,d\omega,
\]
where $|\omega|=1$, $\lambda_{\omega,p}$ is Euclidean hyperplane volume computed in orthonormal coordinates, and $d\omega$ is surface measure with the polar-coordinate normalization; thus $\Rad_2=\Rad$. For odd $n\ge3$ and $u\in C_c^\infty(\R^n)$, local inversion gives \cite{Helgason,Natterer}
\[
 u(x)=\kappa_n(-\Delta)^{(n-1)/2}\Rad_n^*\Rad_nu(x),
\]
where $\Delta$ is the Laplacian and $\kappa_n=1/[2(2\pi)^{n-1}]$. If $\Rad_nu$ vanishes on every hyperplane meeting an open set $D_0$, its backprojection vanishes there, so locality of $\Delta$ gives $u=0$ on $D_0$.

The moment conditions also persist: for compactly supported $u\in L^1(\R^n)$ and integers $j\ge0$,
\[
 \int_\R p^j\Rad_nu(\omega,p)\,dp
 =\int_{\R^n}(x\cdot\omega)^ju(x)\,dx
 \qquad (|\omega|=1).
\]
The right-hand side is a homogeneous polynomial of degree $j$ in $\omega$. A higher-dimensional replacement for the planar $q\in\T$, $z=q\overline z$ geometry and its Hardy model remains to be found.

\section*{Acknowledgments}
The author thanks Mirco Raffaini (University of Milan) for inspiring him to work on this problem and for his continued support, and Jan Boman (Stockholm University) for extensive discussions about the interior Radon problem and for comments on earlier versions of this manuscript.

\end{document}